\documentclass [11pt,paper] {article}
\usepackage[cp866]{inputenc}
\usepackage[english,russian]{babel}
\usepackage{amssymb,amsmath,amsfonts,latexsym,amsthm}
\date{}
\newtheorem {theorem}{\bf Теорема}
\newtheorem {corol}{\bf Следствие}
\newtheorem {lemma}{\bf Лемма}

\newtheorem {pro}{\bf Утверждение}

\def\p0{\parindent 0pt}
\title{{\bf
\normalsize  И.~Ю.~Могильных, Ф.~И.~Соловьева\\
\large ОБ ОТДЕЛИМОСТИ КЛАССА  ГОМОГЕННЫХ   СОВЕРШЕННЫХ ДВОИЧНЫХ
КОДОВ ОТ ТРАНЗИТИВНЫХ}
\thanks{Статья сдана в печать в журнал "Проблемы передачи информации". Работа  выполнена при  финансовой поддержке  гранта Российский Научный Фонд 14-11-00555.}}

\begin{document}

\maketitle

\begin{abstract}
На примере класса  совершенных двоичных кодов доказано
существование двоичных гомогенных нетранзитивных кодов. Тем самым,
с учетом ранее полученных результатов, установлена иерархическая
картина меры линейности  двоичных кодов, а именно имеет место
строгое содержание класса двоичных линейных кодов в классе
двоичных пропелинейных кодов, включающихся строго  в класс
двоичных транзитивных кодов, которые, в свою очередь, строго
содержатся в классе двоичных гомогенных кодов. В работе выводится
критерий транзитивности совершенных двоичных кодов ранга на
единицу больше чем ранг кода Хэмминга той же длины.
\end{abstract}

\section{Введение}

Наиболее близкими по целому ряду свойств  к линейным кодам
(особенно по строению групп автоморфизмов) являются пропелинейные
коды и  транзитивные, все определения см. ниже. Вопрос о
существовании транзитивных кодов, не являющихся пропелинейными,
 был впервые поставлен  в 2006 г. Пухолем, Рифой,
Ф.И.Соловьевой. Позднее, когда была получена классификация
совершенных двоичных кодов длины 15 (см. \cite{OP}) и перечислены
все транзитивные и гомогенные совершенные коды длины 15,
естественно возник вопрос о существовании бесконечной серии
двоичных гомогенных кодов, не являющихся транзитивными. На оба
вопроса  получены положительные ответы.   На первый вопрос ответ
получен в работе \cite{BMRS}, где  доказано, что известный
двоичный код Беста длины 10 с кодовым расстоянием 4, будучи
транзитивным, не является пропелинейным. Существование бесконечной
серии транзитивных непропелинейных совершенных кодов  доказано в
работе \cite{MS2014}:

\begin{theorem} Для любого $n\geq 15$ существуют совершенные двоичные транзитивные коды длины $n$, не являющиеся пропелинейными.
\end{theorem}

Следует отметить, что в \cite{MS2014} было доказано, что  только
один из 201 неэквивалентного  транзитивного совершенного кода
длины 15 является непропелинейным. Ответ на вопрос о существовании
двоичных гомогенных нетранзитивных кодов приводится в настоящей
статье на примере совершенных кодов. Тем самым структура вложения
классов упомянутых выше  кодов, близких к линейным, имеет вид

$$\emph{\textbf{L}} \subset\emph{\textbf{Prl}} \subset \emph{\textbf{Tr}}
\subset\emph{\textbf{Hom}},$$ где $\emph{\textbf{L}}$ -- класс
линейных двоичных кодов, $\emph{\textbf{Prl}}$ -- класс
пропелинейных двоичных кодов, $\emph{\textbf{Tr}}$ -- класс
транзитивных двоичных кодов; $\emph{\textbf{Hom}}$ -- класс
гомогенных двоичных кодов.

Приведем основные определения. Через $F^n$ обозначим  $n$-мерное
метрическое пространство всех двоичных векторов длины  $n$ с
метрикой Хэмминга. Произвольное подмножество векторов
 $C$ из  $F^n$ называется двоичным кодом длины  $n$.
 Код $C$ называется {\it совершенным двоичным кодом} длины $n$, исправляющим одну ошибку, если для любого вектора $x\in F^n$
 найдется единственный вектор  $y$ из $C$ на расстоянии один от  $x$. Без
ограничения общности будем рассматривать только {\it приведенные}
коды, т.е. коды, содержащие нулевое слово  $0^n$ длины $n$ (далее
для краткости будем опускать термины двоичный и приведенный).
Известно, что совершенные двоичные коды с расстоянием  3
существуют тогда и только тогда, когда $n=2^k-1,k>1$. Широко
известно, что  для группы автоморфизмов
  $\mbox{Aut}(F^n)$ пространства $F^n$ справедливо
$$\mbox{Aut}(F^n)= F^n \leftthreetimes S_n = \{(y, \pi) \mid
y \in F^n , \pi\in S_n   \},$$ где $\leftthreetimes$ -- полупрямое
произведение,  $S_n$ --  симметрическая группа подстановок  $n$
координат векторов из  $F^n$. {\it Группой автоморфизмов} Aut$(C)$
произвольного
 кода $C$ длины $n$ называется стабилизатор кода $C$ как множества по группе $\mbox{Aut}(F^n)$, т.е.
$$\mbox{Aut}(C)=\{(y , \pi) \mid  y + \pi(C) =C \}.$$ {\it Группой симметрий} кода $C$ называется множество
$\mbox{Sym}(C)=\{ \pi\in S_n \mid \pi(C)=C \}.$
 Очевидно, что $\mbox{Sym}(C)$ --
 подгруппа группы $\mbox{Aut}(C).$

 Код  $C$
называется {\it транзитивным}, если его группа автоморфизмов
содержит подгруппу, действующую транзитивно на всех его кодовых
словах. Если эта подгруппа регулярна, т.е. ее порядок совпадает с
мощностью кода, то такой код, следуя \cite{RBH1989}, называется
{\it пропелинейным}.  Для транзитивных кодов удобно пользоваться
следующим, эквивалентным приведенному выше, определением: для
каждого кодового слова $y$ из $C$ найдется подстановка
 $\pi$ из $S_n$ такая, что $(y,\pi )\in \, \mbox{Aut}(C),$
 что означает  $y
 + \pi (C)  = C,$ где $\pi $ может не принадлежать группе симметрий
 Sym$(C)$ кода
 $C.$
Многие классы известных кодов являются транзитивными,  см. обзор
результатов, касающихся транзитивных кодов, в параграфе 4  работы
\cite{Sobsor2013}.

{\it Система троек Штейнера} $STS(n)$ порядка $n$ определяется как
система сочетаний из $n$ элементов по три такая, что каждая
неупорядоченная пара элементов содержится в точности в одной
тройке. Известно, что совокупность носителей кодовых слов веса $3$
в любом приведенном двоичном совершен\-ном коде $C$ длины $n$
определяет систе\-му троек Штейнера порядка $n$. Для кодового
слова $y$ кода $C$ через $STS(C,y)$ будем обозначать следующую
систему троек Штейнера $\{supp(x+y): x\in C, d(x,y)=3\}$. Код $C$
называется {\it гомогенным}, если   для любого кодового слова
$y\in C$ система STS$(C,y)$ {\it изоморфна} STS$(C,0^n)$, то есть
найдется подстановка $\pi\in S_n$ такая, что
$\pi(STS(C,y))=STS(C,0^n)$.

Нетрудно видеть, что всякий транзитивный код является гомогенным.

\section{Строение группы вращений кодов Васильева}

{\it Ядром} Ker$(C)$ кода $C$ называется совокупность его {\it
периодов}, т.е. кодовых слов $x\in C$ таких, что $x+C=C.$
Рассмотрим  {\it группу вращений Rot$(C)$} и {\it транслятор
Tr$(C)$} кода $C$: \begin{center}Rot$(C)=\{ \pi\in S_n \mid
\exists y \in C :(\pi, y)\in $ Aut$(C) \},$\end{center}
\begin{center}Tr$(C)=\{ y\in C \mid \exists \pi \in S_n : (\pi,
y)\in $ Aut$ (C) \}.$\end{center} Обозначим через Rot$_y (C)$
 класс смежности в Rot$(C)$, связанный с фиксированным кодовым словом $y\in C$:
\begin{center}Rot$_y(C)=\{ \pi\in S_n  \mid (\pi,
y)\in $ Aut$(C) \}.$\end{center} Ясно, что Rot$_y(C)=\emptyset$
тогда и только тогда, когда $y\notin$ Tr$(C).$ Имеет место
следующее свойство, связывающее Rot$(C)$ и Tr$(C)$: \,\,
|Aut$(C)|=|$Sym$(C)|\cdot |$Tr$(C)|=|$Rot$(C)|\cdot |$Ker$(C)|.$
Легко показать справедливость следующих утверждений.
\begin{pro}\label{Rot,Sym} Для любого двоичного кода $C$ выполняется $$\mbox{Sym}(C) \leq\mbox{Rot}(C) \leq  \mbox{Sym(Ker}(C)).$$
\end{pro}


В работе \cite{ASH2005} исследована группа симметрий кодов
Васильева. Для получения основного результата данной статьи нам
потребуется изучить группу вращений кодов Васильева. Оказалось,
что ряд результатов, справедливых для группы симметрий кодов
Васильева, имеет место для группы вращений. Напомним необходимые
определения из \cite{ASH2005}.


{\it Линейной $i$-компонентой} (далее кратко {\it
$i$-компонентой}) $R_i^n$ будем называть линейную оболочку троек
кода $C$ длины $n$, содержащих $i$, $i \in \{1,2,\ldots,n\}$.
Заметим, что в случае кода Хэмминга длины $n$, $R_i^n$ является
его подкодом.

Пусть $C$ -- произвольный совершенный код длины $n$, $n=2^k-1$,
 $\lambda : C \rightarrow \{0,1 \}$ -- произвольная функция, удовлетворяющая
 $\lambda (0^n)= 0$. Рассмотрим коды $C_\lambda
=\{(y,\lambda(y),0^n)  \mid y\in C \}$  и
 $R_{n+1}^{2n+1}=
\{(x, |x|, x) \mid \, x\in F^n \},$ где $|x|=x_1+ \ldots +
x_n(\bmod \, 2).$ Оба кода имеют длину $2n+1$, код
$R_{n+1}^{2n+1}$ является $(n+1)$-компонентой. Пользуясь кодами
$C_\lambda $  и $R_{n+1}^{2n+1}$, определим двоичный совершенный
код Васильева \cite{V62}:
 \begin{equation}\label{Vasilievconstruction} V_C^{\lambda}= C_\lambda + R_n^{2n+1}=\{(x + y, |x| + \lambda
 (y), x) \mid \, x\in F^n, y\in C\} \end{equation}
длины $2n+1$.

Заметим, что в силу того, что компонента $R_{n+1}^{2n+1}$ является
подпространством ядра кода Васильева, то для любого $y\in
V_C^{\lambda}$ и $v\in R_{n+1}^{2n+1}$ верны следующие
соотношения:
$$STS(V_C^{\lambda},y)=STS(V_C^{\lambda},y+v),\,\,
y\in Tr(C) \mbox{ тогда и только тогда, когда } y+v \in Tr(C).$$

   Обозначим через $t_i$ транспозицию,
переводящую  $i$ в  $i+n+1$, $i\in I,$ где $I=\{1,2,\ldots,n\}.$
Для вектора $u\in F^n$ рассмотрим подстановку $\tau _u=\prod
_{i\in supp(u)} t_i$. Обозначим совокупность всех подстановок,
соответствующих $F^n$, через $G$, т.е. $G=\{ \tau _u \mid u\in
F^n\}.$ Для произвольной подстановки
$$\pi=\left(\begin{array}{cccc} 1 & 2 & \ldots & n \\ \pi(1) & \pi(2) &
\ldots  & \pi(n)
\end{array}\right)$$
 подстановка  $\sigma_{\pi}$, называемая  {\it дубликатором}, определяется следующим
 образом:
 $$\sigma_{\pi}=\left(\begin{array}{ccccccccc} 1 & 2 & \ldots & n & n+1 &  n+2 & n+3 & \ldots & 2n+1 \\
\pi(1) & \pi(2) & \ldots  & \pi(n) & n+1 &  \pi(1)+n+1 &
\pi(2)+n+1 & \ldots  & \pi(n)+n+1
\end{array}\right).$$

 Множество всех дубликаторов обозначается через
$D$, т.е. $D=\{ \sigma_{\pi} \mid \pi \in S_n\}.$
{\it Стабилизатор} $i$-ой координаты  группы вращений кода $C$
 обозначим через $\mbox{St}_{i}(\mbox{Rot}(C)).$
 \begin{lemma}\label{Lemma1}  Для любого кода Васильева  $V_C^{\lambda}$  справедливо
  $\mbox{St}_{n+1}(\mbox{Rot}(V_C^{\lambda}))\leq D \rightthreetimes
  G.$\end{lemma}
  Доказательство этой леммы аналогично доказательству предложения 3 из  \cite{ASH2005}, достаточно заменить $Sym(V_C^{\lambda})$ на $Rot(V_C^{\lambda})$ и использовать утверждение \ref{Rot,Sym}.
Нам потребуется также следующая лемма (см. предложение 4 в работе
\cite{ASH2005}):
\begin{lemma}\label{equal} 
Для произвольных кода $C$ длины $n$ и определенной выше функции
$\lambda$
 выполняется  $\tau _u
((y,\lambda (y), 0^n))= (y,\lambda (y), 0^n) + (y\ast u,0,y\ast
u)$ для любого $y \in C$, где   $y\ast u =(y_1 u_1,\ldots ,y_n
u_n).$
 \end{lemma}




\begin{theorem}\label{rotn+1}  Пусть $V_C^{\lambda}$ -- произвольный код
Васильева, $z=(y^{\prime},\lambda(y'),0^{n})\in V_C^{\lambda}$.
Подстановка $\rho_{z}$ принадлежит
$\mbox{St}_{n+1}\mbox{(Rot}_{z}(V_C^{\lambda}))$ тогда и только
тогда, когда она может быть представлена композицией
$\rho_{z}=\sigma_{\pi_{y^{\prime}}} \circ \tau_u$ для некоторых
$\pi_{y^{\prime}} \in \mbox{Rot}_{y^{\prime}}(C)$ и $u\in F^n$
таких, что для любого $y\in C$  выполняется соотношение
\begin{equation}\label{Compatibily} \lambda (y^\prime) + \lambda
(y) +\lambda(y^\prime + \pi_{y^{\prime}}(y))=u\, \cdotp
y,\end{equation} где $u\, \cdotp y$ --  скалярное произведение
векторов  $u$ и $y$ из $F^n$. \end{theorem} Заметим, что при
$y^\prime=0$ справедливо
$\mbox{St}_{n+1}(\mbox{Rot}_{y^{\prime}}(V_C^{\lambda}))=
\mbox{St}_{n+1}(\mbox{Sym}(V_C^{\lambda}))$, равенство
(\ref{Compatibily}) преобразуется в равенство
$\lambda(y)+\lambda(\pi(y))=u\cdot y$ (см. определение
$\lambda$-согласованности в \cite{ASH2005}).

\begin{proof}
 Рассмотрим произвольную подстановку $\rho_{z}\in \mbox{St}_{n+1}\mbox{(Rot}_{z}(V_C^{\lambda}))$, где
$z=(y',\lambda(y'),0^n)$. По лемме \ref{Lemma1} имеем
$\rho_{z}=\sigma_{\pi_{y^{\prime}}} \circ \tau _u$, где $u$ --
некоторый вектор $F^n$, а $\pi_{y^{\prime}}$ некоторая подстановка
из $S_n$. Покажем, что $\pi_{y^{\prime}}\in Rot_{y'}(C)$ и для
всех $y$ выполнено равенство (\ref{Compatibily}).

Так как $z+\rho_z(V_C^{\lambda})=V_C^{\lambda}$, то для любого
$y\in C$ имеем
$$(y',\lambda(y'),0^n)+\sigma_{\pi_{y^{\prime}}}\circ \tau
_u(y,\lambda(y),0^{n})\in V_C^{\lambda}.$$
 Согласно Лемме \ref{equal} и определению дубликатора, имеем $\sigma_{\pi_{y'}}\circ \tau
_u(y,\lambda(y),0^{n})=(\pi_{y'}(y+u\ast y),\lambda(y),
\pi_{y'}(u\ast y))$. Следовательно для всех $y\in C$ вектор\\
\noindent
 $(y^{\prime},\lambda(y^{\prime}),0^n)+\sigma_{\pi_{y^{\prime}}}\circ \tau
_u(y,\lambda(y),0^{n})=$\\
$(y^{\prime},\lambda(y'),0^n)+(\pi_{y^{\prime}}(y),\lambda(y)+u\cdot
y,0^n)+(\pi_{y^{\prime}}(y\ast u),u\cdot y,\pi_{y^{\prime}}( y\ast
u))=$\\
$(y^{\prime}+\pi_{y^{\prime}}(y),\lambda(y')+\lambda(y)+u\cdot
y,0^n) + (\pi_{y^{\prime}}(y\ast u),u\cdot y,\pi_{y^{\prime}}(
y\ast u))$

\noindent принадлежит $V_C^{\lambda}$. В силу того, что код
Васильева $V_C^{\lambda}$ представляет собой некоторую
совокупность классов смежности по компоненте $R_{n+1}^{2n+1}$, то
прибавление вектора $(\pi_{y^{\prime}}(y\ast u),u\cdot
y,\pi_{y^{\prime}}( y\ast u))$, принадлежащего $R_{n+1}^{2n+1}$, к
вектору
$(y^{\prime}+\pi_{y^{\prime}}(y),\lambda(y^{\prime})+\lambda(y)+u\cdot
y,0^n)$ никак не влияет на свойство последнего принадлежать коду
$V_{C}^{\lambda}$ для любого $y$ из кода $C$. Отсюда заключаем что
$\pi_{y^{\prime}}$ принадлежит $Rot_{y^{\prime}}(C)$, а равенство
(\ref{Compatibily}) выполняется для всех $y$.
\end{proof}

Заметим, что в работе \cite{KP2014} Д.С.\,Кротов и В.Н.\,Потапов
предположили, что транзитивные коды ранга $n-log(n+1)+1$ следует
искать в классе кодов Васильева с функцией, удовлетворяющей
некоторому равенству, эквивалентному равенству
(\ref{Compatibily}), однако объяснения этому факту приведено не
было. Также выполнение равенства (\ref{Compatibily}) при
$\pi_{y^{\prime}}=id$ для всех $y, y'$ эквивалентно определению
квадратичной функции, рассмотренной в той же работе.

  \begin{corol} Пусть $\lambda$ -- нелинейная булева функция на коде Хэмминга $H$. Тогда
$y' \in Tr(V_H^{\lambda})$ тогда и только тогда, когда
 найдутся  $\pi \in \mbox{Sym}(H)$, $u\in F^n$ такие, что для всех $y\in H$ выполнено $\lambda (y^{\prime}) + \lambda (y) +
\lambda(y^{\prime} + \pi_{y^{\prime}}(y))=u\, \cdotp y.$

\end{corol}
\begin{proof}
Пусть $\rho_z\in \mbox{Rot}_z(V_H^{\lambda})$. Заметим, что код,
полученный из $V_H^{\lambda}$ удалением $j$-й координаты будет
линейным тогда и только тогда, когда $j=(n+1)/2$. Следовательно,
рассматривая равенство $z+\rho_z(V_H^{\lambda})=V_H^{\lambda}$,
приходим к выводу
$\mbox{Rot}_z(V_H^{\lambda})=St_{n+1}(\mbox{Rot}_z(V_H^{\lambda}))$,
что в силу теоремы \ref{rotn+1} дает требуемое.
\end{proof}

\section{Гомогенные совершенные коды ранга $n-log(n+1)+1$}
В этом разделе докажем существование бесконечной серии гомогенных
нетранзитивных совершенных кодов длины $n$ для каждого допустимого
$n\geq 15$ ранга $n-log(n+1)+1$. Построение этих кодов базируется
на существовании гомогенных нетранзитивных совершенных кодов длины
15.

Для дальнейшего нам потребуется  конструкция системы троек
Штейнера Ассмуса и Маттсона \cite{AM} и ее связь с конструкцией
Васильева, напомним их. Пусть $S$ является STS$(n)$ и
$\theta:S\rightarrow\{0,1\}$ -- произвольная булева функция на
тройках $S$. Определим
 $S^{\theta}$ -- систему STS$(2n+1)$ следующим образом:
 \begin{itemize}
    \item тройки $\{i,n+1,i+n+1\}$ принадлежат $S^{\theta}$ для
    любого $i\in\{1,\ldots,n\};$
    \item если $\theta(\{i,j,k\})=0$, то тройки $\{i,j,k\}, \{i,j+n+1,k+n+1\}, \{k,i+n+1,j+n+1\}, \{j,i+n+1,k+n+1\}$ принадлежат $S^{\theta};$
    \item если $\theta(\{i,j,k\})=1$, то тройки $
    \{i+n+1,j+n+1,k+n+1\}, \{i,j,k+n+1\}, \{j,k,i+n+1\}, \{i,j,k+n+1\}$ принадлежат $S^{\theta}$.
\end{itemize}


\begin{pro}\label{prop} Пусть $C$ является совершенным кодом,
$\lambda:C\rightarrow\{0,1\}$, $z=(x+y,|x|,x)$ является кодовым
словом кода $V_C^{\lambda}$. Тогда $STS(V_C^{\lambda},z)$ есть
$STS(C,y)^{\theta}$, где
$\theta(supp(y+y'))=\lambda(y)+\lambda(y')$, для $y'\in C$, таких
что $d(y',y)=3$.
\end{pro}

\subsection{Гомогенные нетранзитивные совершенные коды длины 15}

Для $n=15$  были исследованы все совершенные коды ранга 12,
оказалось, что среди них существует всего два гомогенных
нетранзитивных совершенных кода -- это коды, обозначаемые
$V22^{1}$ и $V3^{1}1$ согласно классификации С.~А.~Малюгина
\cite{Mal} двоичных совершенных кодов длины 15, полученных
свитчингами из кода Хэмминга той же длины.

Произвольное кодовое слово $x$ будем  задавать его носителем
$supp(x)=\{i|x_i=1\}$. Пусть $H$ -- код Хэмминга длины 7,
порожденный векторами $\{1,2,3\},\{1,4,5\},\{1,6,7\},\{2,4,6\}.$

Код $V22^{1}$ --  код Васильева $V_H^{\lambda}$, см.
(\ref{Vasilievconstruction}), где
$\lambda(0^7)=\lambda(\{1,6,7\})=\lambda(\{1,3,5,7\})=\lambda(1^7)=0$,
на остальных кодовых словах кода $H$ значение функции $\lambda$
равно 1.

Код $V3^{1}1$ -- это код Васильева $V_H^{\lambda}$, где
$\lambda(0^7)=\lambda(\{1,6,7\})=\lambda(\{2,4,6\})=\lambda(\{4,5,6,7\})=0$,
на прочих кодовых словах кода $H$ значение  $\lambda$ равно 1.

\begin{lemma}  \label{HomoV3^{1}1}  Коды $V3^{1}1$ и $V22^{1}$ являются  гомогенными.
   \end{lemma}
\begin{proof}
Заметим, что определенные выше функции $\lambda$ на коде $H^7$
обладают следующим свойством: для любого кодового слова $y\in H^7$
выполняется $|\{y'\in H^7: d(y',y)=3,\lambda(y)=\lambda(y')\}|$
принимает следующие значения:
1, 2, 5, 6. 
 Принимая во внимание утверждение \ref{prop} покажем,
что $S^{\theta}$ изоморфна $S^{\theta'}$, где $\theta$ и $\theta'$
-- произвольные булевы функции на системе троек Штейнера $S$
порядка 7 c 1, 2, 5 или 6 нулями.

Вначале докажем утверждение для функций с одним нулем или двумя
нулями.
 Если $\theta$ и $\theta'$ имеют одинаковое число нулей, то в силу 2-транзитивности
группы симметрий системы троек Штейнера порядка 7 (как двоичного
кода) найдется подстановка $\pi$, переводящая нули $\theta$ в нули
$\theta'$. Тогда подстановка $\sigma_{\pi}$ переводит тройки
$S^{\theta}$ в тройки $S^{\theta'}$.

Пусть $\theta$ имеет один ноль, а $\theta'$ имеет два  нуля. В
силу сказанного выше, без ограничения общности все тройки, на
которых  функции $\theta$ и $\theta'$ принимают значение ноль,
являются тройками системы  $S$, которые имеют один общий элемент,
скажем $i$.

 Заметим, что
 тройки $\{i,j,k\}, \{i,j+8,k+8\}, \{k,i+8,j+8\},
\{j,i+8,k+8\}$ переходят в
  тройки $\{i+8,j+8,k+8\}, \{i,j,k+8\}, \{j,k,i+8\}, \{i,j,k+8\}$ с помощью транспозиции $t_i=(i,i+8)$. Так как эти совокупности троек есть в точности совокупности,
    индуцируемые тройкой $\{i,j,k\}$ в
конструкции Ассмуса и Маттсона при $\theta(\{i,j,k\})=0$ или 1, то
подстановка $(i,i+8)$ переводит все тройки $S^{\theta}$,
содержащие $i$, в тройки $S^{\theta'}$, содержащие $i+8$.
Следовательно, $t_i(S^{\theta})=S^{\theta'}$. 

Случай когда $S^{\theta}$ является системой троек, где $\theta$
принимает 5 или 6 нулей, сводится к рассмотренному выше случаю с
одним или двумя нулями для функции $\theta$ с помощью подстановки
$\tau_{1^7}$, которая меняет местами координаты $i$ и $i+8$ для
любого $i\in\{1,\ldots,7\}$. При этом система троек
$\tau_{1^7}(S^{\theta})$ получается из $S$  применением функции c
одним нулем или двумя нулями.
\end{proof}

\begin{lemma}  \label{TrV3^{1}1}  Коды $V3^{1}1$ и $V22^1 $ не являются транзитивными.
 \end{lemma}
\begin{proof}  Напомним, что $Rot(H)=Sym(H)$. Докажем, что для кода $V3^{1}1$ нарушается
равенство (\ref{Compatibily}), где $\pi$ пробегает $Sym(H)$, а
$y^{\prime}=1^7$.
 Поскольку $\lambda(1^7)=1$, условие
(\ref{Compatibily}) преобразуется в следующее
\begin{equation}\label{CompatibilyV3^{1}1}  \lambda(1^7 + \pi (y))=u\, \cdotp y + \lambda (y) + 1.\end{equation}
 Рассмотрим правую часть (\ref{CompatibilyV3^{1}1}). Заметим, что равенство $u\, \cdotp y =0$ задает гиперпространство $U=\{y:y \cdotp u =0\}$ в коде Хэмминга $H$ длины 7, число таких гиперпространств равно числу ненулевых точек в $H$, т.е. равно 15. Справедливы следующие случаи:\\
 a)  вектор $u$ принадлежит коду $H^{\bot}$ -- ортогональному к коду $H$;\\
b) имеем 7 подпространств -- $i$-компонент кода $H$, $i\in
\{1,2,\ldots,7\}$, имеющих весовое распределение, состоящее из
нулевого вектора,
трех кодовых слов веса 3, трех кодовых слов веса 4 и единичного кодового слова  кода $H$,  например $R_1^7=<\{1,2,3\},\{1,4,5\},\{1,6,7\}>$;\\
 c) имеем 7 подпространств, содержащих  нулевой вектор, четыре кодовых слова веса 3 и три кодовых слова веса 4  кода $H$, например
 $<\{1,6,7\},\{2,4,6\},\{4,5,6,7\}>.$

 В случае а) в силу $u\, \cdotp y =0$ соотношение (\ref{CompatibilyV3^{1}1}) примет вид $\lambda(1^7 + \pi (y))= \lambda (y) + 1,$ из которого следует, что функция $\lambda$ имеет одинаковое число 0 и 1, что противоречит ее определению.

  В случае b) снова исследуем соотношение  (\ref{CompatibilyV3^{1}1}).
Без ограничения общности полагаем, что $R_1^7\subset U$.
Рассмотрим обе части равенства (\ref{CompatibilyV3^{1}1}) при $y$,
пробегающих множество кодовых слов веса 3 в коде $H$. С одной
стороны, $1^7+\pi(y')$ пробегает множество векторов веса 4,
следовательно $\lambda(1^7+\pi(y'))$ равно 0 только в одном случае
и 1 в шести. С другой стороны, $\lambda (y) + 1$ будет равно 0 на
пяти кодовых словах, 1 -- на двух кодовых словах, $u\, \cdotp y
=0$ для трех кодовых слов $\{1,2,3\},\{1,4,5\},\{1,6,7\}$.
Следовательно, для кодовых слов $y$ таких, что $\lambda(y)=0$,
должно необходимо выполняться $u\, \cdotp y =0$. Аналогичное
утверждение получим в случае, когда $y$ пробегает множество
кодовых слов веса 4 в коде $H$: $\lambda(1^7 + \pi (y))$ принимает
значение 0 на двух кодовых словах и 1 на пяти. Другими словами,
все нули функции $\lambda(y)$, заданные на коде $H$, должны
принадлежать $U$, т.е. случай b) невозможен.

Для каждого из 7 подпространств в случае c) снова рассмотрим
вектор $y$, пробегающий множество кодовых слов веса 3 в коде $H$ и
убеждаемся, что $u\, \cdotp y + \lambda (y) + 1$ будет равно 0 как
минимум на двух кодовых словах, в то время как левая часть
$\lambda(1^7 + \pi (y))$ принимает значение 0 только на одном
кодовом слове и 1 -- на шести кодовых словах, противоречие.

Нетранзитивность кода $V22^1$ доказывается аналогично. \end{proof}

\section{Бесконечная серия гомогенных нетранзитивных кодов}

Покажем, что с помощью конструкции  Васильева из гомогенных
совершенных кодов можно получать гомогенные коды большей длины.

\begin{theorem}\label{VasilievHomogenious}  Если $C$ -- произвольный гомогенный совершенный код, то код
Васильева $V_C^{\lambda}$ при $\lambda\equiv 0$ является
гомогенным. \end{theorem}
 \begin{proof} Обозначим код Васильева $V_C^{\lambda}$ при $\lambda\equiv 0$ через $V_C^{0}$. Пусть $z, z'\in V_C^{0}$. Рассмотрим две системы
 $STS(V_C^{\lambda},z)=STS(C,y)^{0}$ и
 $STS(V_C^{\lambda},z')=STS(C,y')^{0}$, где $z=(y+x,|x|,x)$, $z'=(y'+x',|x'|,x')$.
 Пусть $\pi(STS(C,y))=STS(C,y')$. Тогда несложно видеть, что
 дубликатор $\sigma_{\pi}$ подстановки $\pi$ переводит
 $STS(C,y)^{0}$ в $STS(C,y')^{0}$. Действительно,
$\sigma_{\pi}(\{i,n+1,i+n+1\})=\{\pi(i),n+1,\pi(i)+n+1\},$
$\sigma_{\pi}(\{i,j,k\})=\pi (\{i,j,k\})\in STS(C,y')$,
$\sigma_{\pi}(\{i+n+1,j+n+1,k\})=\{\pi(i),\pi(j)+n+1,\pi(k)+n+1\}\in
STS(C,y')^{0}$.
  \end{proof}


Компоненту $R_j^n$ кода Хэмминга $H^{n}$ длины $n$ назовем {\it
протыкающей} нули и единицы функции $\lambda$, если   $R_j^n$
содержит и нули и единицы функции $\lambda$.

Рассмотрим булеву функцию $\lambda$, определенную на коде Хэмминга
$H^{n}$ длины $n$. Рекуррентно определим функцию $\lambda_N$ на
коде $H^N$, положив $\lambda_n\equiv\lambda$ и для $y\in
H^{(N-1)/2}$
$$\lambda_N((y,0^{(N+1)/2)}) +
R^N_{(N+1)/2})=\lambda_{(N-1)/2}(y).$$ Пусть $C$ -- совершенный
код длины $N$, полученный из кода $V_{H^{n}}^{\lambda}$ длины
$2n+1$ посредством $s$-кратного применения конструкции Васильева с
нулевой функцией,
$s=\log(N+1)-\log(2n+2)$ (далее для кода Васильева важно будет
указывать, какой длины код Хэмминга и какая функция использовались
для построения этого кода). Тогда верно следующее сведение

\begin{lemma}\label{pierce} 1. Код $C$ длины $N$ эквивалентен коду
Васильева $V_{H^{(N-1)/2}}^{\lambda_{(N-1)/2}}$.

 2. Если $R^{(N-1)/2}_j$ протыкает нули и единицы функции $\lambda_{(N-1)/2}$, то $R^{N}_j$ и $R_{j+(N+1)/2}^N$
 протыкают нули и единицы $\lambda_N$.
\end{lemma}
\begin{proof} 1. Код $C$ длины $N$ имеет ранг на единицу больше, чем ранг кода Хэмминга той же длины.
Известно, что всякий такой код эквивалентен коду Васильева
$V_{H}^{\lambda_{(N-1)/2}}$ для некоторой функции $\lambda$ и
некоторого кода Хэмминга $H$ длины $(N-1)/2$. Однако для
доказательства теоремы  \ref{iterative_Homo} нам важно убедиться,
что исходный код Хэмминга длины $n$ содержится как подкод в коде
Хэмминга последней итерации при построении кода $C$. При $s=1$
имеем $N=4n+3$ и подстановка

$$\varphi=\prod_{i=0}^{n}(n+i,3n+2+i,2n+2+i),$$
\normalsize являющаяся произведением циклов длины 3, фиксирующая
каждую из первых $(n-1)/2$ координатных позиций, переводит код
$V_{H^n}^{\lambda}$ в код $C$.
Для $s>1$ аналогичным образом, используя индукцию по $s$, можно
выписать соответствующую подстановку.

2. Рассмотрим код Хэмминга $H^{(N-1)/2}$. Пусть $j\leq (N-1)/2$,
$y\in R^{(N-1)/2}_j$. Тогда так как $(y,0^{(N+1)/2})\in R^{N}_j$,
то функция $\lambda_N$ является $R^{N}_j$-протыкающей.

Пусть $y=\sum_{m=1,\ldots,r}\{i_m,j,k_m\}\in R_j^{(N-1)/2}$. Тогда
вектор
$y'=y+\sum_{m=1,\ldots,r}\{k_m,(N+1)/2,k_m+(N+1)/2\}+\{j,(N+1)/2,j+(N+1)/2\}$
принадлежит компоненте $y+R^N_{(N+1)/2}$, следовательно
$\lambda_N(y')=\lambda_N(y)$. С другой стороны,
$y'=\sum_{m=1,\ldots,r}\{i_m,k_m+(N+1)/2,j+(N+1)/2\}$ принадлежит
компоненте $R^N_{j+(N+1)/2}$ и следовательно $\lambda_N$ содержит
 и нули и единицы в компоненте $R^N_{j+(N+1)/2}$, так как  обладает таким свойством в $R^N_{j}$.
\end{proof}

\begin{theorem} \label{iterative_Homo}
Пусть $\lambda$ -- нелинейная функция, протыкающая компоненты,
 число нулей и единиц $\lambda$ различны. Тогда
$Tr(V_{H^{N}}^{\lambda_{N}})=
Tr(V_{H^{(N-1)/2}}^{\lambda_{(N-1)/2}})+R_{(N+1)/2}^{N}$.
\end{theorem}
\begin{proof}
В силу определения, функция $\lambda_{N}$ не является линейной,
следовательно  cогласно Следствию 1 для выполнения  $y'\in
Tr(V_{H^{N}}^{\lambda_{N}})$ необходимо и достаточно, чтобы
нашлись подстановка $\pi\in S_{N}$ и вектор $u\in F^N$ для любого
$y\in H^N$ такие, что

\begin{equation}\label{Treq}
\lambda_N(y' + \pi_{y^{\prime}} (y))=\lambda_N(y')+ \lambda_N (y)
+ u\, \cdotp y.
\end{equation}


Рассмотрим  линейное пространство
$L=R_{n+1}^{2n+1}+R_{2n+2}^{4n+3}+\ldots+R_{(N+1)/2}^N.$ Заметим,
что $L$ также может быть представлено  как следующая линейная
оболочка компонент (равенство пространств легко показывается
взаимным включением друг в друга с учетом определения
$i$-компонент):
\begin{equation}
\label{Leq} L=\langle R^N_{n+1},R^N_{2n+2},\ldots
R^N_{(N+1)/2}\rangle .\end{equation}

Пространство $U=\{y\in H^N: u\cdotp y =0\}$ либо совпадает с кодом
Хэмминга $H^N$, либо является его гиперплоскостью, то есть
подпространством на единицу меньшей размерности. Заметим, что
всякая гиперплоскость кода Хэмминга $H^N$ пересекается с любым
линейным подкодом кода $H^N$ либо по половине его векторов, либо
является его надпространством.

 Отсюда имеем следующие случаи:

\begin{enumerate}
    \item $|U\cap L|=|L|/2;$
    \item $L\subset \{y\in
H^{N}: u\cdotp y =0\}.$
\end{enumerate}


Случай 1. Пусть $|U\cap L|=|L|/2$. Если $y$ пробегает произвольный
класс смежности $a+L$ по подкоду  $L$ в равенстве (\ref{Treq})
тогда, так как $\lambda_N$ постоянна на $L$ и, следовательно
постоянна на любом классе смежности по подкоду $L$,  левая часть
равенства (\ref{Treq}) принимает ровно половину нулей и  единиц
для  класса смежности $a+L$. Таким образом, функция $\lambda$
имеет одинаковое число нулей и единиц на коде
$C=V_{H^{(N-1)/2}}^{\lambda_{(N-1)/2}}$. Противоречие.

Случай 2. Пусть $L\subset U$. Тогда любой класс смежности $a+L$
является подмножеством $U$ или не пересекается с ним. Рассмотрим
равенство (\ref{Treq}) при условии, что  $y$ пробегает класс
смежности $\pi^{-1}_{y^{\prime}}(y')+R^N_{i}$, где $i$ кратно
$n+1$. В силу равенства (\ref{Leq}) и условия $L\subset U$, имеем

$$\lambda_N(\pi_{y^{\prime}}(R^N_{i}))=\lambda_N(R_{\pi_{y^{\prime}}(i)}^{N})=\lambda_N(\pi^{-1}_{y^{\prime}}(y'))+\lambda_N(y')+u\cdot \pi^{-1}_{y^{\prime}}(y').$$
 Другими словами, $\lambda_N$ принимает
постоянные значения на $R_{\pi_{y^{\prime}}(i)}^{N}$.

 Заметим, что в силу леммы \ref{pierce}, компонента $R_j^N$ протыкает  нули и единицы функции
 $\lambda_N$ тогда и только тогда, когда
  $j$ не кратно $(n+1)$.
 Отсюда заключаем, что $\pi_{y^{\prime}}(\{1\leq j\leq N: j\equiv 0 (mod (n+1))\})=(\{1\leq j\leq N: j\equiv 0 (mod
 (n+1))\})$, так как в противном случае $\lambda_N$ не будет
 постоянной на $R^{\pi(i)}_{N}$.

 Итак, учитывая равенство (\ref{Leq}), получаем $\pi(L)=L$. Поскольку код $C$ разбивается на классы смежности по
 $L$, представителями которого являются кодовые слова $H^n$,
 то подстановка $\pi$ переставляет эти классы смежности по $L$.

 Рассмотрим представителей классов смежности по $L$,
 являющихся кодовыми словами веса 3 кода  $H^n$. Заметим, что в классе
 смежности $a+L$ существуют векторы веса 3 тогда и только тогда, когда вектор $a$ имеет вес
 3. Таким образом подстановка $\pi_{y^{\prime}}$ переставляет классы смежности
 $\{i,j,k\}+L$, где  $\{i,j,k\}\in H^n$. Другими словами, найдется
 подстановка $\sigma_{y^{\prime}}$ из группы симметрий системы троек Штейнера
 кода $H^n$ такая, что
$$\pi_{y^{\prime}}(\{i,j,k\}+L)=\sigma_{y^{\prime}}(\{i,j,k\})+L.$$

  Откуда  в силу того, что группа симметрий
 кода Хэмминга совпадает с группой симметрий его системы троек Штейнера, заключаем, что $\pi_{y^{\prime}}(y+L)=\sigma_{y^{\prime}}(y)+L$
 для любого $y\in H^n.$
Тогда выполняется последовательность  равенств:
$\lambda_N(y'+\pi_{y^{\prime}}(y+L))=\lambda_N(y'+\sigma_{y^{\prime}}(y)+L)=\lambda(y'+\sigma_{y^{\prime}}(y)).$

Отсюда, подставляя $y$ и $y'$ из $H^n$ в равенство (\ref{Treq}) и
предполагая, что $u=(u',u'')$, где $u'\in F^n, u''\in F^{N-n}$,
получим
$$\lambda(y'+\sigma_{y^{\prime}}(y))=\lambda(y)+\lambda(y')+u'\cdot y,$$
что влечет  $y'\in Tr(V_{H^{n}}^{\lambda})$ согласно Теореме
\ref{rotn+1}.
\end{proof}

Из теоремы \ref{iterative_Homo}, лемм   \ref{HomoV3^{1}1}  и
\ref{TrV3^{1}1} получаем основной результат данной работы.

\begin{theorem}
Для любого $n\geq 15$ существуют  двоичные совершенные гомогенные
коды длины $n$, не являющиеся транзитивными.
\end{theorem}

{\bf Замечания.} Заметим, что для кодов Моллара имеет место
теорема, аналогичная теореме \ref{VasilievHomogenious} (см.,
например, определение кода Моллара в \cite{Sobsor2013}). Коды
Моллара являются обобщением кодов Васильева, теорема несложно
доказывается с учетом строения системы троек Штейнера кода
Моллара, которая подробно описана, например, в работе
\cite{MS2014Mollard}.

\begin{theorem}\label{MollardHomogenious}  Если $C$ -- произвольный гомогенный совершенный код длины $m$, $H$ -- код Хэмминга длины $r$, то код Моллара  $M(C,H)$ длины $mr + m + r$ при $f\equiv 0$ является гомогенным. \end{theorem}
Таким образом,  гомогенные совершенные коды можно строить,
используя конструкцию Моллара, хотя следует отметить, что
технически осуществить это будет существенно сложнее изложенного
выше.

\end{document}